\documentclass{amsart}
\usepackage[mathscr]{eucal}
\usepackage{amsmath,amssymb,hyperref}
\usepackage{amstext}
\usepackage{graphics}
\usepackage{xcolor}
\usepackage{mathrsfs}
\usepackage{epstopdf}

\usepackage{tikz}
\usepackage{tikz-cd}
\usepackage[toc,page]{appendix}
\usetikzlibrary{matrix}

\usepackage[all]{xy}

\usepackage[T1]{fontenc}
\usepackage[utf8]{inputenc}

\newtheorem{thm}{Theorem}[section]
\newtheorem{THM}{Theorem}

\newtheorem{cor}[thm]{Corollary}
\newtheorem{prop}[thm]{Proposition}

\newtheorem{lemma}[thm]{Lemma}

\theoremstyle{definition}

\DeclareMathOperator{\Aut}{Aut}

\DeclareMathOperator{\id}{id}

\def\Diff{\mathrm{Diff}(\mathbb C,0)}
\def\Difffor{\widehat{\mathrm{Diff}}(\mathbb C,0)}

\def\C{\mathbb C}

\def\F{\mathcal F}

\def\X0{X^{\circ}}

\def\Y0{Y^{\circ}}

\input xy
\xyoption{all}

\addtocounter{section}{0}             
\numberwithin{equation}{section}       

\begin{document}

\title[Holonomy representation of quasi-projective leaves]
{Holonomy representation of quasi-projective leaves of codimension one foliations}

\author[B. Claudon]{Beno\^{\i}t CLAUDON}
\address{IECL, Universit\'e de Lorraine, BP 70239, 54506 Vand{\oe}uvre-l\`es-Nancy Cedex, France}
\email{benoit.claudon@univ-lorraine.fr}

\author[F. Loray]{Frank Loray}
\address{Univ Rennes, CNRS, IRMAR - UMR 6625, F-35000 Rennes, France}
\email{frank.loray@univ-rennes1.fr}

\author[J.V. Pereira]{Jorge Vit\'{o}rio PEREIRA}
\address{IMPA, Estrada Dona Castorina, 110, Horto, Rio de Janeiro,
Brasil}
\email{jvp@impa.br}

\author[F. Touzet]{Fr\'ed\'eric Touzet}
\address{Univ Rennes, CNRS, IRMAR - UMR 6625, F-35000 Rennes, France}
\email{frederic.touzet@univ-rennes1.fr}

\subjclass[2010]{37F75; 14F35}
\keywords{Quasi-projective groups; Holonomy representations.}

\thanks{The third author is supported by CNPq and FAPERJ,
and BC, FL and FT benefit from support of CNRS and ANR-16-CE40-0008 project ``Foliage''}

\begin{abstract}
We prove that a representation of the fundamental group of a quasi-projective manifold into the group of formal diffeomorphisms of one variable either
is virtually abelian  or, after taking the quotient by its center, factors through an orbicurve.
\end{abstract}

\maketitle

\setcounter{tocdepth}{1}
\sloppy

\section{Introduction}

\subsection{Statement of the main result}
Let $\Difffor$ be the group of formal biholomorphisms of $(\mathbb C,0)$. The purpose of this article is  to present a proof of the following result.

\begin{THM}\label{THM:A}
    Let $X$ be a quasi-projective manifold and $\rho : \pi_1(X) \to \Difffor$ a representation.
    Suppose $G=\mbox{Im}\ \rho$ is not virtually abelian, then its center $Z(G)$ is necessarily
    a finite subgroup and the induced representation ${\rho}': \pi_1(X) \to G/Z(G)$ factors through an orbicurve.
\end{THM}

In the particular case where $X= \overline{X}$ is a projective manifold, this result appears as Theorem D of  \cite{leaves}. As a matter of fact, in the compact case, the result is also proved (loc.cit.) for compact K\"ahler manifolds.

\subsection{Context} Representations of fundamental groups of quasi-projective manifolds in $\Diff \subset \Difffor$ appear as holonomy representations of algebraic leaves
 of codimension one holomorphic foliations. There is a conjecture, formulated by Cerveau, Lins Neto and others \cite{croco1}, on the structure of codimension one  foliations on projective manifolds of dimension at least three which predicts that
they admit a singular  transversely projective structure (see \cite{croco5} for a precise definition) or contain a subfoliation of codimension two by algebraic leaves.  Theorem \ref{THM:A} is in accordance with this conjecture, and is potentially useful to investigate it. 


\subsection{Strategy of proof} We split the proof of Theorem \ref{THM:A} into two different parts. The first part deals with representations having infinite linear part. The strategy is the same as the one carried out in \cite{leaves}.
 The second part considers representations with finite linear part. In this second part, we either reduce to the compact case after a finite ramified covering, or we exploit the structure of the representation at a neighborhood of infinity in order to construct the fibration using a result from \cite{Totaro}, see also \cite{jvpJAG},  similarly to what has been done in \cite[Theorem A]{croco5} to describe representations of fundamental groups of quasi-projective manifolds in $SL(2,\mathbb C)$ which are not quasi-unipotent at infinity.

\section{Representations with infinite linear part}

\subsection{Monodromy of group extensions}
If a group $G$ is the extension of a group $H$ by a group $N$, i.e. if $G$ fits into the short exact sequence of groups
\begin{equation}\label{E:short exact group}
    1 \to N \to G \to H \to 1 \, ,
\end{equation}
we have a natural group morphism from $H$ to the automorphisms of the abelianization of $N$
\begin{align*}
    H & \longrightarrow \Aut\left(\frac{N}{[N,N]} \right) \\
    h & \longmapsto \lbrace [n] \mapsto \hat{h} [n] \hat{h}^{-1} \rbrace
\end{align*}
where $\hat h$ is any element in $G$ mapping to $h$. The image $\Gamma$ of $H$ into  $\Aut\left(\frac{N}{[N,N]} \right)$ will be
called the monodromy of the group of extension (\ref{E:short exact group}).

\begin{lemma}\label{L:group}
Let $\Gamma$ and $\Gamma'$ be the respective  monodromies
of the two group extensions  $ 1 \to N \to G \to H \to 1$  and $ 1 \to N' \to G' \to H' \to 1$.
If there exist  surjective morphisms $\alpha : N \to N'$, $\beta : G \to G'$, and $\gamma : N \to N'$  fitting
into the commutative diagram
$$\xymatrix{
1 \ar[r] & N \ar[r] \ar[d]^\alpha & G \ar[r] \ar[d]^\beta & H \ar[d]^\gamma \ar[r] & 1 \\
	1 \ar[r] & N' \ar[r] & G' \ar[r] & H' \ar[r] & 1
}$$
then we have a natural surjective morphism from  $\Gamma$ to $\Gamma'$.
\end{lemma}
\begin{proof}
Let $\rho : H \to \Gamma \subset \Aut(N/[N,N])$ and $\rho' : H ' \to \Gamma' \subset \Aut(N'/[N',N'])$ be the
monodromy representations of two exact sequences. In order to produce a surjective morphism from $\Gamma$ to $\Gamma'$
it suffices to show that any element $h \in \ker \rho$ is mapped to the identity by the composition $\rho' \circ \gamma : H \to
\Aut(N'/[N',N'])$.

Let  $h \in \ker \rho$ be an arbitrary element and consider a lift $\hat h$ to $G$. By assumption
\[
    [n] = \hat h [n] \hat h^{-1}
\]
for any $[n] \in N/[N,N]$. Applying $\beta$ to this identity we deduce that $\beta(\hat h)$ acts trivially on the image of morphism
$[\alpha] : N/[N,N] \to N'/[N',N']$ induced by $\alpha$. Since the abelianization functor is right exact we deduce that $(\rho' \circ \gamma) (h) = \id$
as wanted.
\end{proof}

Consider now the Zariski closure of $\Gamma$ inside  of the linear\footnote{Here we implicitly assume that $N$ is finitely generated.} algebraic group $\Aut_{\mathbb C}(N/[N,N] \otimes \mathbb C)$ and call it $\Gamma_{\mathbb C}$.  The naturalness of the surjective morphism $\Gamma \to \Gamma '$ gives the following consequence.

\begin{cor}\label{C:surjective}
	Under the assumptions of Lemma \ref{L:group}, we have a natural surjective morphism of linear algebraic groups $\Gamma_{\mathbb C} \to \Gamma^{'}_{\mathbb C}$.
\end{cor}

\subsection{Semi-simplicity}

The result below is a particular case of a more general result by Deligne, see \cite[Corollary 4.2.9]{Deligne}.

\begin{thm}\label{T:Deligne}
	Let $X$ and $B$ be quasi-projective manifolds. Assume $B$  endowed with a base point $b \in B$.
	Let $f: X \to B$ be a  morphism such that $R^n f_* \mathbb Q$ is a local system over $B$ for every non negative integer $n$. Let $G$ be the
	Zariski closure of the image of $\pi_1(B,b)$ in $\Aut_{\mathbb C}( (R^nf_* \C)_b)$, and let $G^0$ be the connected
	component of the identity of $G$. Then:
	\begin{enumerate}
		\item If $f$ is proper,  then $G^0$ is semi-simple.
		\item In general, the radical of $G^0$ is unipotent.
	\end{enumerate}
\end{thm}

Recall that the radical of a  linear algebraic group is the largest connected solvable normal subgroup. In particular, a Lie group is semi-simple if, and only if, its radical is trivial.

\subsection{Lifting factorizations}

Let $\rho : \pi_1(X,x) \to \Difffor$ be a representation. For $k \in \mathbb N$, let us denote by $\rho_k : \pi_1(X,x) \to J^k \Difffor$ the composition of $\rho$ with the natural projection/truncation $\Difffor \to J^k \Difffor$ onto the group of $k$-jets of diffeomorphisms.

\begin{prop}\label{P:keyinfty}
Let $X$ be a quasi-projective manifold and let $\rho : \pi_1(X,x) \to \Difffor$ be a non-abelian representation. If $\rho_1 : \pi_1(X,x) \to \mathbb C^* (= J^1 \Diff)$ has infinite image and factors through a (non necessarily proper) morphism $f: X \to C$  with connected fibers, then $\rho$ also factors through $f$.
\end{prop}
\begin{proof}
Up to shrinking $X$ with respect to the Zariski topology,  we can assume that $f : X \to C$ is a topological fiber space over a non-proper algebraic curve $C$. In order to prove that $\rho$ factors through $f : X\to C$, it suffices to prove that $\rho_k$ has the same property for an arbitrary natural number $k$.

Let $k$ be smallest integer for which the factorization of $\rho_k$  through  $f$ does not hold and, aiming at a contradiction, let us consider the following commutative  diagram.
$$\xymatrix{
1 \ar[r] & \pi_1(F) \ar[r] \ar@{>>}[d] & \pi_1(X) \ar[r]^{f_*} \ar@{>>}[d]^{\rho_k} \ar@{>>}[dr]^{\rho_{k-1}}& \pi_1(C) \ar@{>>}[d] \ar[r] & 1 \\
1 \ar[r] & \rho_k(\pi_1(F))  \ar@{^{(}->}[d] \ar[r] & \rho_k(\pi_1(X))  \ar@{^{(}->}[d] \ar[r] & \rho_{k-1}(\pi_1(X)) \ar[r]  \ar@{^{(}->}[d] & 1\\
1 \ar[r] & (\mathbb  C,+) \ar[r] & J^k \Difffor \ar[r] & J^{k-1} \Difffor \ar[r] & 1
}$$
The top row is nothing but the homotopy sequence for fibrations: as we are assuming that $C$ is non-proper we have that $\pi_2(C)=0$.  On the bottom row, we have used the isomorphism between the kernel $J^k \Diff_{k-1}$ of the canonical projection  $J^k \Diff  \twoheadrightarrow J^{k-1} \Diff$ and $(\mathbb C,+)$.

Let $\Gamma \subset \Aut\left( H_1(F,\mathbb Z) \right)$ be  the monodromy group of the top row, and $\Gamma' $ be the monodromy group of the middle row. According to Theorem \ref{T:Deligne}, the   Zariski closure $G$ of $\Gamma$ in $\Aut(H_1(F,\mathbb C))$  has quasi-unipotent radical. In particular, $G$ has no (algebraic) surjection to $\mathbb C^*$.
On the other hand since we are assuming that $\rho_1(\pi_1(X))  \subset {\mathbb C}^*$ is infinite, we have that $\Gamma'$  is isomorphic to a Zariski dense subgroup of $\mathbb C^*$. These two facts contradict Corollary \ref{C:surjective}, showing that there is no such smallest $k$. We conclude that the representation $\rho$ factors through $f$.
\end{proof}

\subsection{Synthesis} Theorem \ref{THM:A} for representations with infinite linear part 
follows from the result below.

\begin{thm}\label{T:infinitelinear}
Let $X$ be a quasi-projective manifold and $\rho : \pi_1(X) \to \Difffor$ a representation. Suppose that the image of $\rho_1$ is infinite. If $\rho$ is not abelian  then there exists a finite \'etale Galois covering $\pi: Y \to X$, a morphism $f: Y \to C$, and  a representation $\psi: \pi_1^{orb} (C ) \to \Difffor$ such that the diagram
$$\xymatrix{
\pi_1(Y)   \ar[r]^{f_*} \ar[d]_{\pi_*} & \pi_1^{orb}(C) \ar[d]^{\psi}\\
		\pi_1(X) \ar[r]^{\rho}  & \Difffor
}$$
commutes.
\end{thm}
\begin{proof}
After replacing $X$  by a suitable \'{e}tale Galois covering $Y$, we can assume that the linear part of $\rho$ has  torsion free image. We still denote by $\rho$ the induced representation of $\pi_1(Y)$ in $\Difffor$.
Let $\gamma_0\in \pi_1(Y)$ such that $\lambda_{\gamma_0}$ has infinite order (here $\lambda_{\gamma_0}$ denotes the linear part of $\rho(\gamma_0)$). Then, after performing a suitable conjugation in $\Difffor$, one can assume that $\rho (\gamma_0)=\lambda_0 z$ (\cite [Theorem 5.1 and references therein] {leaves}). 
Let  $m\geq 2$ be the first positive number such that $\rho_m : \pi_1(X) \to J^m \Difffor$ has non abelian image. It is equivalent to say that, for every $\gamma\in\pi_1 (X)$, $\rho_m (\gamma)(z)= \lambda_\gamma z +a_\gamma z^m$ with $\gamma\to a_\gamma$ a non identically zero map. Indeed, the fact that $\rho_{m-1} (\gamma)(\lambda_{\gamma_0}z)=\lambda_{\gamma_0}\rho_{m-1} (\gamma)(z)$ for any $\gamma\in\pi_1(X)$ can be rewritten in the following way: $\rho_{m-1} (\gamma)(z)=\lambda_\gamma z$. Since $\rho_m(\pi_1(X))$ is not abelian, we infer that $\rho_m (\gamma)$ has the form above with $a_\gamma$ not identically zero.

In particular, ${ \rho_1}^{\otimes {1-m}}$, the $(1-m)$-th power of the linear part of $\rho$  possesses a nontrivial affine extension, namely $$\gamma\mapsto \big(a_\gamma \lambda_\gamma ^{-m},\lambda_\gamma ^{1-m}\big)\in \C\rtimes{\C}^*=\mbox{Aff}(\C)$$ i.e $H^1(X,{\rho_1}^{\otimes {1-m}})\not=0$. It follows from a result by Arapura \cite{Arapura} later refined in \cite[Theorem 1]{Bartolo} (see also \cite[Theorem 3.1]{CousinPereira})
that there exists a surjective morphism $f$ from $X$ to an orbicurve $C$  such that   $\rho_1^{\otimes{1-m}}$  factors
through $f_* : \pi_1(Y) \to \pi_1^{orb}(C)$. Since we are assuming that $\rho_1$ has torsion free image, we deduce that the linear part of $\rho$ also factors through $f_*$. Since $\rho_1$ is infinite, Proposition \ref{P:keyinfty} concludes the proof.
\end{proof}

\section{Representations with trivial linear part}

\subsection{Subgroups of $\Diff$ tangent to the identity}
For  $k \in \mathbb N$, we will denote by $\Difffor_k$ the subgroup of $\Difffor$ composed 
of the formal biholomorphisms which are tangent to the identity up to order $k$. Therefore
$\Difffor_0 = \Difffor$ and $\Difffor_1$ is the subgroup of elements with trivial linear part.

We recall the characterization of maximal abelian groups of $\Difffor_1$
(\cite[\S 1.4]{Frankpseudo}).

 \begin{thm}\label{T:formalG}
Let $G \subset \Difffor_1$ be a subgroup. If $G$ is abelian, then
there exists $ {\varphi} \in \Difffor$ such that $ \varphi_* G$ is a subgroup of one of 

 $\mathbb E_{k,\lambda} = \{ f(z) = \exp\left(t v_{k,\lambda}\right); , t \in \mathbb C\}$, for some $k \in \mathbb N^*$ and $\lambda \in \mathbb C$ and 
\[v_{k,\lambda} = \frac{z^{k+1}}{1+ \lambda {z^k}} \frac{\partial}{\partial z}
\]

\end{thm}


\begin{lemma}\label{L:normalform}
If $f \in \Difffor_1$ is different from identity,  then there exits a unique one-dimensional vector space $V$ of formal meromorphic
$1$-forms preserved by $f$. Moreover,  $f^* \omega= \omega$ for every $\omega \in V$.
\end{lemma}
\begin{proof}
Let $f \in \Difffor_1$ be an element different from the identity. According to \cite[Proposition 1.3.1]{Frankpseudo} there exist $\varphi \in \Difffor$,  $k \in \mathbb N$, and $\lambda \in \mathbb C$
such that $f = \varphi^{-1} \circ \exp(v_{k,\lambda}) \circ \varphi$ 
It turns out that the formal meromorphic $1$-form
\[
    \omega = \varphi^* \left(  \frac{dz}{z^{k+1}}   + \lambda \frac{dz}{z} \right)
\]
is preserved by $f$.  Let now $\omega'$ be another meromorphic $1$-form such that
$f^* \omega' = \mu \omega'$ for some $\mu \in \mathbb C^*$. Since $\omega'= h \omega$ for some $h \in \mathbb C((z))$, it follows
that  $f^* h = \mu h$. Comparing Laurent series we deduce that $h \in \mathbb C^*$ and  $\mu=1$. Therefore $V = \mathbb C \omega$
is the unique one dimensional vector space of formal meromorphic $1$-forms preserved by $f$.
\end{proof}

\begin{lemma}\label{L:preservesV}
If $G \subset \Difffor_1$ is a subgroup which preserves a one-dimensional
vector space $V$ of formal meromorphic $1$-forms, then $G$ is an abelian subgroup.
\end{lemma}
\begin{proof}
If $G$ is not abelian, then there exist elements $f, g \in G$ of different orders of tangency to the identity say $k_f$ and $k_g$.
Therefore the $1$-forms associated with them have (see the proof of Lemma \ref{L:normalform}) poles of order $k_f +1$ and $k_g+1$ and cannot belong to the same one dimensional
vector space.
\end{proof}

\begin{lemma}\label{L:normal}
Let $G \subset \Difffor_1$ be a subgroup which contains a non-trivial (i.e. different from the identity)  abelian normal subgroup $H$.
Then there exists a non trivial (formal) meromorphic $1$-form $\omega = \sum_{i=-k}^{\infty} a_i z^i dz$, unique up to multiplication in ${\C}^*$, 
such that every element $g \in G$ satisfies $g^* \omega = \omega$. In particular, $G$ itself is abelian and
contained in a subgroup of $\Difffor_1$ isomorphic to $(\mathbb C,+)$.
\end{lemma}
\begin{proof}
Let $f \in \Difffor_1$ be an element different from the identity. Let $V = \mathbb C \cdot \omega$ be the unique
one dimensional vector space of meromorphic $1$-forms preserved by $f$.
The centralizer of $f$ coincides with the subgroup of $\Difffor_1$ with elements satisfying
identity $h^* \omega = \omega$, see  \cite[Proposition 1.3.2]{Frankpseudo}.

Let now $h \in H$ be a nontrivial element of the abelian normal subgroup $H$ of $G$. Let $g \in G$ be
an arbitrary element. Since $H$ is normal, we have that $g \circ h = h' \circ g$ for some $h' \in H$ distinct from the identity. Let $\omega$ be a non-zero meromorphic $1$-form fixed by every element of $H$. Therefore
\[
    (g \circ h)^* \omega = (h' \circ g )^* \omega \implies h^* (g^* \omega) = g^* \omega \, .
\]
It follows that $g^* \omega$ is a  constant multiple of $\omega$ (as a matter of fact, since $g$ is tangent to $\id$, $g^* \omega = \omega$).
Thus $g$ is in the centralizer of $H$ as claimed.  

Being abelian, $G$ is isomorphic by Theorem \ref{T:formalG} to a subgroup of  $\mathbb E_{k,\lambda} \simeq \C$. This concludes the proof.
\end{proof}

\subsection{Representations at a neighborhood of a connected divisor}
Let $D = \sum_{i=1}^k  D_i \subset M$ be a compact connected simple normal crossing hypersurface with irreducible components $D_i$ on a smooth complex manifold $M$ of dimension $m$. Let $\rho: \pi_1(X,q) \to \Difffor_1$ be a representation where $X=M-D$. By the classical suspension process, one can construct a $m+1$ dimensional formal neighborhood $\hat X$ of $X$ carrying a smooth codimension $1$ (formal) foliation $\F$ having $X$ as a leaf and having $\rho$ as holonomy representation along $X$. If $U$ is an open subset of $X$, $\hat U$ will denote the restriction of $\hat X$ over $U$.

\begin{lemma}\label{L:abelian}
With the notations above, assume that $\rho(\gamma_i) \neq \id$ for every
$\gamma_i$ corresponding to small  loops around irreducible components of $D$. Then, there exists a neighborhood $U$ of $D$ such that the restriction of the representation  $\rho$ to $U-D$ has abelian image.
\end{lemma}
\begin{proof}
For each $i$,  let $U_i$ be a small tubular  neighborhood of
$D_i$ and set $\displaystyle{U_i^{\circ} = U_i - \cup_{j} D_j}$,  $\displaystyle{U= \cup_i  U_i} $. Note that $U_i^{\circ}$ has the homotopy type of a $S^1$-bundle over $\displaystyle{D_i^{\circ}:=D_i-\cup_{i\not=j} D_j}$
and therefore the subgroup generated by $\gamma_i$ in $\pi_1(U_i^{\circ})$ is normal.
 
By Lemma \ref{L:normal}, $\rho\big(\pi_1( {U_i}^{\circ})\big)$ preserves pointwise a unique one dimensional vector space $V_i$ of  meromorphic  $1$-forms in $(\mathbb C,0)$. Equivalently, the foliation $\F$ restricted to $\widehat{{U_i}^{\circ}}$ is defined by a \textit{closed} meromorphic formal one form $\omega_i$ with pole on ${U_i}^{\circ}$, unique up to multiplication in ${\C}^*$. Set $W_i= \C \omega_i$.
To analyze what happens  at a non-empty  intersection $D_i \cap D_j$, $i\not=j$,  we can assume that both $\gamma_i$ and $\gamma_j$ have base points
near $D_i\cap D_j$. Thus  $\gamma_i$ commutes with
$\gamma_j$. From Lemma \ref{L:normalform} and the assumptions on $\rho(\gamma_i)$, one deduces that $W_i=W_j$ on  ${U_i}^{\circ}\cap {U_j}^{\circ}$. This implies that the restriction of $\F$ to $\widehat {U-D}$ can be defined by a rank one local system of \textit{closed} meromorphic one forms. In other words, the holonomy group of $\F_{|\widehat {U-D}}$ evaluated with respect to some transversal $(T,q)\simeq\widehat{(\C,0)}$ ($q\in U-D$) preserves a one dimensional vector space of formal meromorphic one forms. Lemma \ref{L:preservesV} gives the sought conclusion.
\end{proof}

\begin{cor}\label{C:negative}
Notations and assumptions as in Lemma \ref{L:abelian}. Assume moreover that $M$ is a complex surface.  If $D_1, \ldots, D_k$ are the irreducible components of $D$, then the intersection matrix $(D_i \cdot D_j)$ is not negative definite.
\end{cor}
\begin{proof}
Aiming at a contradiction, assume  that the intersection matrix is negative definite. Let $U$ be a neighborhood of $D$ as in Lemma \ref{L:abelian}. Assume also that $U$ has the same homotopy type as $D$. On the one hand, the class of  any of the  loops $\gamma_i$ in $H_1(U-D, \mathbb Z)$ is torsion, see \cite[page 11]{Mumford} or \cite[Proposition 3.5]{croco5}.  On the other hand, since the representation is abelian by Lemma \ref{L:abelian} and with values in the torsion free group $\Difffor_1$, the assumption $\rho(\gamma_i) \neq \id$ implies that the class of $\gamma_i$ in $H_1(U-D, \mathbb Z)$, the abelianization of $\pi_1(U-D)$, is of infinite order. This gives the sought contradiction and establishes the corollary.
\end{proof}

\begin{cor}\label {C:generalnegative}
Notations and assumptions as in Lemma \ref{L:abelian}. Assume moreover that $M$ is a quasiprojective manifold of dimension $m\geq 2$ and $\overline{M}\subset {\mathbb P}^N$ a smooth compactification. Let $H\subset M$ be a hyperplane section.  If $D_1, \ldots, D_k$ are the irreducible components of $D$, then the intersection matrix $(D_i \cdot D_j\cdot H^{m-2})$ is not negative definite.
\end{cor}

\begin{proof}
The case $m=2$ has been already settled in Corollary \ref{C:negative}. If $\dim X \geq 3$ then, by \cite{HammLe},  the general hyperplane section of $X=M-D$ has the fundamental group isomorphic to the original one. This establishes the proof, thanks to Corollary \ref{C:negative}.
\end{proof}

\subsection{Factorization}

The proof of the factorization result for representations with trivial linear part is adapted from the proofs of \cite[Theorem 3.1 and Theorem A]{croco5}.

\begin{thm}\label{T:triviallinear}
Let $X$ be a quasi-projective manifold of dimension $m\geq 2$ and $\rho : \pi_1(X) \to \Difffor$ a representation. Suppose that $\rho$ is not virtually abelian and has finite linear part, then the conclusion of Theorem \ref{T:infinitelinear} holds true.
\end{thm}
\begin{proof}
Up to passing to an \'etale finite cover, one can firstly assume that $\rho$ has trivial linear part. Let $\overline{X}$ be a projective manifold compactifying $X$ such that $\overline{X} - X$ is a simple normal crossing divisor. If the representation $\rho$ can be extended to a representation of $\pi_1(\overline{X})$ to $\Difffor$, then the result follows from \cite[Theorem D]{leaves}. If instead the representation does not extend to $\pi_1(\overline{X})$, then  let $E$ be the minimal divisor contained in $\overline{X}- X$ for which the representation extends to $\pi_1(\overline{X}-E)$. In particular, $\rho(\gamma)\neq \id$ for any small loop winding around a component of $E$.
	
Let $D = \sum D_i$ be a connected component of $E$. According to Lemma \ref{L:abelian} the restriction of the representation to a neighborhood $U$ of $D$ is abelian. Moreover, by Corollary \ref{C:generalnegative}, the intersection matrix $(D_i,D_j):=(D_i \cdot D_j\cdot H^{m-2})$ is not negative definite. 
	
Notice that a finitely generated subgroup $G$ of $\Difffor$ is residually finite. Indeed, $G$ is obtained as the inverse limit of $G^m$, its truncation up to order $m$, which is clearly linear. If moreover, $G \subset \Difffor_1$, non abelianity of $G$ is equivalent to non solvability \cite[Remark 5.9]{leaves}. In particular, the subgroups appearing in the derived sequence $(G^{(n)})_{n\ge 0}$ of $G$ are not trivial for any $n\ge0$. Coming back to our setting, take $G= \rho(\pi_1(\overline{X}-E))$ and $S\subset G$ the subgroup defined as $S=\rho(\pi_1(U-E))$, where $U$ is a neighborhood of $E$ such that $S$ is abelian (whose existence is guaranteed by Lemma \ref{L:abelian}). Let $a\in G^{(2)}$ be a non trivial element in the second derived group of $G$. We can produce a surjective morphism  $q : G \to F$ to a  finite group $F$ such that $q(a)\neq 1$.  Note that $q(S)$ has index at least three in $q(G)$: if the index is two, the group $F$ has to be metabelian ($F^{(2)}=1$),
 but $1\neq q(a)\in F^{(2)}$. After taking a resolution of singularities of the ramified covering of $\overline X$ (\'etale over $X$) determined by $q \circ \rho : \pi_1(\overline{X}-E) \to F$, we end up with a situation similar to the initial one with the advantage that now, the boundary divisor has at least three distinct connected components. We keep the original notation. 
	
Hodge index Theorem implies that the intersection matrix of each of the components of $E$  is semi-negative definite, and in particular, each one of them is the support of  an effective divisor with self-intersection zero. Hodge index Theorem also implies that all these divisors with zero self-intersection have proportional Chern classes. We are in position to apply \cite[Theorem 2.1]{Totaro} (see also \cite[Theorem 2]{jvpJAG}) in order to produce a fibration $f: \overline{X} \to C$ over a curve $C$ with connected fibers which contracts the boundary divisor to points. 
	
Let $F$ be a fiber of $f$ contained in a sufficiently small neighborhood of one of the connected components of the boundary. It follows from Lemma \ref{L:abelian} that $\rho(\pi_1(F))$ is abelian. Since we are assuming that $\rho$ is not abelian, Lemma \ref{L:normal} implies that $\rho(\pi_1(F))= \id$. This proves the result. 
\end{proof}

\section{Proof of Theorem \ref{THM:A}}
Assume  that the image of $\rho : \pi_1(X) \to \Difffor$ is not virtually abelian and that, after a Galois \'etale covering $\pi:Y \to X$,  we have the factorization of $\rho$ through a morphism $f:Y \to C$ to an orbicurve $C$, as in the conclusion of Theorems \ref{T:infinitelinear} and \ref{T:triviallinear}.

If $F$ denotes a  general fiber of $f$, ${\pi}^* \rho$ is trivial in restriction to $F$ and also to $\alpha (F)$ for any deck transformation $\alpha\in \mathrm{Gal}(\pi)$. On the other hand, ${\pi}^* \rho$ has infinite image. This implies that the group of deck transformations of $\pi$  preserves the fibration, otherwise there would exist $\alpha \in \mathrm{Gal}(\pi)$ such that $f$ maps $F_\alpha:=\alpha (F)$ onto a dense open Zariski subset of $C$. This implies that the index of $f_* (\pi_1 (F_\alpha))$ is at most finite in $\pi_1(C)$: a contradiction.   We can descend the fibration to a fibration 
$g: X \to C'$, where the orbicurve $C'$ is a finite quotient of $C$ under the natural action of the group of deck transformations of $\pi$.

By construction, the restriction of $\rho$ to the fundamental group of  fibers of $g$ have finite image $A$ in $\Difffor$. In particular, it is conjugated to a finite group of rotations. Moreover, it follows from the homotopy sequence of  fibrations that $A$ is a finite normal subgroup of $\Gamma = \rho(\pi_1(X))$. Since linear part is preserved by conjugation it implies that $A$ is in the center of $\Gamma$. Therefore the composition of $\rho$ with the natural quotient morphism $\Gamma \to \Gamma/ Z(\Gamma)$ factors through $g$. \qed

\bibliographystyle{amsplain}

\end{document}